\documentclass[12pt]{article}
\usepackage{amsfonts}
\usepackage[latin1]{inputenc} \RequirePackage[T1]{fontenc}
\usepackage[english]{babel} \RequirePackage{amsmath}
\usepackage{amssymb} \RequirePackage{amsthm}
\usepackage{amsmath}
\usepackage{pslatex}
\usepackage{fancyhdr}
\usepackage{comment}
\usepackage{graphicx}
\usepackage{geometry}
\numberwithin{equation}{section}
\usepackage{color}

\newtheorem{theo}{Theorem}[section]
\newtheorem{lem}{Lemma}[section]
\newtheorem{rem}{Remark}[section]

\newtheorem{defi}{Definition}[section]
\newtheorem{propo}{Proposition}[section]

{\penalty-20\null\hfill$\square$\par\medbreak}

\newcommand{\R}        {{{\rm I\! R}}}

\newcommand{\E}        {{{\rm I\! E}}}

\newcommand{\N}        {{{\rm I\! N}}}
\newcommand{\Q}        {{{\rm I\! Q}}}

\renewcommand{\P}        {{ {\rm I \hskip -2pt P}}}

\begin{document}
	\title{Asymptotic behavior for delayed backward stochastic differential equations\thanks{This work is partially supported by Emerging Regional center of excellence ERCE 2021-2025 of "Laboratoire de Mathématiques et Applications}}
	
	\author{Clément Manga $^{a}$ \thanks{cmaga@univ-zg.sn}\;\; Auguste Aman $^{b}$ \thanks{aman.auguste@ufhb.edu.ci, corresponding author}\;\; 
	 and\; Navègue Tuo $^{c}$\thanks{tnavegue@yahoo.fr} \\ 
		a. UFR Sciences et Technologie, Université Assane Seck, Ziguinchor, Sénégal\;\;\;\;\;\;\;\;\;\\
		b.	UFR Mathématiques et Informatique, Université Félix H. Boigny, Abidjan,\;\;\;\;\;\;\;\\ Côte d'Ivoire\\
		c. UFR Sciences Economiques et Gestions, Université Alassane Ouattara, Bouaké,\\ Côte d'Ivoire}
	
	\date{}
	\maketitle\thispagestyle{empty} \setcounter{page}{1}
	
	\thispagestyle{fancy} \fancyhead{}
	\fancyfoot{}
	\renewcommand{\headrulewidth}{0pt}
	\begin{abstract}
This paper is devoted to study the asymptotic properties for the solution of decoupled forward backward stochastic differential equations with delayed generator. As an application, we establish a large deviation principe for solution of the backward equation.
	\end{abstract}

	\vspace{.08in}\textbf{MSC}: 34F05; 60H10; 60F10; 60H30,60G07 \\
	\vspace{.08in}\textbf{Keywords}: Forward stochastic differential equations; Backward Stochastic differential equations; Delayed generators; Large deviation principe, Meyer-Zheng topologie.
	
	\section{Introduction}
	
The theory of large deviations concerns the asymptotic behaviour of remote tails of sequences probability distributions. While some basic ideas of the theory can be traced to Laplace theorie, the formalization started with insurance mathematics, namely ruin theory with Cramèr \cite{Cramer}. A unified formalization of large deviation theory was developed by Varadhan \cite{Varadhan}, Donsker and Varadhan \cite{DV1,DV2,DV3}. This theory formalizes the heuristic ideas of concentration of measures and widely generalizes the notion of convergence of probability measures.
Subsequently several issues concerning large deviation principles and their applications to stochastic differential equations (SDEs, for short) were studied in many works. We can  cite for e.g. Freidlin and Wentzell \cite{FW}, Stroock \cite{St}, Deuschel and Stroock \cite{DS}, Den Hollander \cite{DH}, Diédhiou and Manga \cite{DM}, and others. Next, with probability methods, by the contraction principle, the same small random perturbation for BSDEs and the Freidlin-Wentzell's large deviation estimates in $\mathcal{C}([0, T ], \R^n)$ are also obtained by \cite{Essay, FR, Sophi, KM} and  references therein.

On the other hand, there is little work on the large deviations theory for delayed SDE. The first one was du by M. Scheutzow \cite{S} within the context of
additive white noise. The multiplicative case with particular dependence on the history have be done by Mohammed and Zhang in \cite{MZ}. Recently, Manga and Aman (see \cite{Manga-AMAN}) established the large deviation principle in for delayed SDE in the general case.

However, to our knowledge the there exist no work in literature concerning the large deviation principle for the delayed BSDEs introduced in the last decay.

Indeed, in their works \cite{DI1, DI2}, Delong and Imkeller introduce two equations with dynamics given, for $t\in [0, T]$, respectively by
\begin{eqnarray}
Y(t)&=& \xi+\int^{T}_{t}f(s,Y_s,Z_s)ds-\int_{t}^{T}Z(s)dW(s)\label{Eq2}
\end{eqnarray}
and
\begin{eqnarray}
Y(t)&=& \xi+\int^{T}_{t}f(s,Y_s,Z_s,U_s)ds-\int_{t}^{T}Z(s)dW(s)-\int_{t}^{T}U(s,z)\widetilde{M}(ds,dz).\label{Eq21}
\end{eqnarray}
This type of equations has been called BSDE with time delayed generator. It is driven respectively by a Brownian motion and  a Lévy process, the components of which are given by a Brownian motion and a Poisson random measure. In this new type of equations, a generator $f$ at time $s$ depends in some measurable way on the past values of a solutions $(Y_s, Z_s) = (Y(s + u), Z(s + u))_{-T\leq u\leq 0}$ and $(Y_s, Z_s,U_s) = (Y(s + u), Z(s + u),U(s+u))_{-T\leq u\leq 0}$ respectively. In this two papers, existence and uniqueness results have been obtained if only if the Lipschitz constant or terminal time are small enough. Further, several solution properties are investigated, including the comparison principle, measure solutions, the inheritance property of boundedness from terminal condition to solution, as well as the BMO martingale property for the control component. In addition, in \cite{DI2}, author was studied the Malliavin's differentiability of the solution of a time delayed BSDE, both with respect to the continuous component of the Lévy process, which coincides with the classical Malliavin derivative for Hilbert-valued random variables, as well as with respect to the pure jump part, in terms of an increment quotient operator related to Picard's difference operator. they proved also that the well-known connection between $(Z,U)$ and the Malliavin trace of $Y$ still holds in the case of time delayed generators. 

In addition to its theoretical importance, we would like to point out that the interest for BSDEs with time delayed generators also lies in their applications in reinsurance and finance theory (see \cite{D1}). 
More precisely, let us consider the financial market which consists of two tradeable instruments. A risk-free
asset whose price $B = (B(t), 0 \leq t \leq T)$ is given by equation
\begin{eqnarray*}
\frac{dB(t)}{B(t)}=r(t) dt, \ \ B(0)= 1,
\end{eqnarray*}
where $r = (r(t), 0 \leq t \leq T)$ denotes the risk-free interest rate in the market. The price $D = (D(t), 0 \leq  t \leq  T)$ of a risky bond with maturity T is given by
\begin{eqnarray*}
\frac{dB(t)}{B(t)}= (r(t) + \sigma(t)\theta(t)) dt + \sigma(t) dW(t), \ \ D(0) = d_0.
\end{eqnarray*}
Setting $\mu(t) = r(t) + \sigma(t)\theta (t)$, then the investment portfolio $X = (X(t), 0 \leq t \leq T)$ and an amount $\pi = (\pi (t), 0 \leq t \leq T)$ invested in the bond $D$, satisfy the equation
\begin{eqnarray*}
dX(t) = \pi(t)(\mu(t) dt + \sigma(t) dW(t)) + (X(t)- \pi(t))r(t) dt, \ \ X(0) = x.
\end{eqnarray*}
By the change of variables (actualization principle), we set, for  $0 \leq t \leq T$,  $Y(t) = X(t) e^{-\int_0^t r(s) ds}$ and $Z(t)= \sigma (t) \pi(t) e^{-\int_0^t r(s) ds}$. Therefore, the discounted portfolio process  $Y = (Y(t), 0 \leq t \leq T)$ satisfies under the measure $\Q$ this dynamic:
\begin{eqnarray*}
	dY(t) = Z dW^{\Q}(t), \ \ Y(0) = y,
\end{eqnarray*}
where $W^{\Q}$ is a $\Q$-Brownian motion. Consider a terminal liability or an investment target $\xi$. We deal with the
problem of finding an investment strategy $Z$ and an investment portfolio $Y$ which replicate a liability or meet
a target $\xi(Y(T) , Z(T) )$ depending on the applied strategy or the past values of the portfolio. Our financial problem
is equivalent to deriving a solution $(Y, Z)$ to the time-delayed BSDE
\begin{eqnarray*}
	Y(t) = \xi(Y_T , Z_T) - \int_t^T Z dW^{\Q}(t)
\end{eqnarray*}
 which follows immediately from (1.2). For example, the dynamic of option-based portfolio assurance is the following time-delayed BSDE:
 \begin{eqnarray*}
Y(t)= Y(0) + (Y(T) - Y(0))^+ - \int_t^T Z dW^{\Q}(t). 	
 \end{eqnarray*}
For more details, the reader is encouraged to consult Delong \cite{D1} and El Karoui, Jeanblanc and Lacoste \cite{KJL}.

Let now consider this system of disturbed decoupled delayed forward backward stochastic differential equations (FBSDEs, in short)
\begin{eqnarray}\label{Eq3}
X^{\varepsilon}(t)&=& x+\int^{t}_{0}b(s,X^{\varepsilon}_s)ds+\sqrt{\varepsilon}\int_{0}^{t}\sigma(s,X^{\varepsilon}_s)dW(s)\nonumber\\\\
Y^{\varepsilon}(t)&=& g(X^{\varepsilon}_T)+\int^{T}_{t}f(s,X^{\varepsilon}_s,Y^{\varepsilon}_s,Z^{\varepsilon}_s)ds-\int_{t}^{T}Z^{\varepsilon}(s)dW(s).	\nonumber
\end{eqnarray}
We know, thanks to the work of Manga and Aman \cite{Manga-AMAN}, that the solution of the forward equation converges in probability, when $\varepsilon$ tends to $0$, to the solution $\mathcal{X}$ of the deterministic equation
\begin{eqnarray*}
	\mathcal{X}(t)&=& x+\int^{t}_{0}b(s,\mathcal{X}_s)ds.
\end{eqnarray*}
and satisfies a large deviation principle.

The objectif of this paper is to show that the solution $(Y^{\varepsilon},Z^{\varepsilon})$ of the
backward equation converges, when $\varepsilon$ tends to $0$, to $(\mathcal{Y},0)$, where $\mathcal{Y}$ is the solution of the deterministic equation
\begin{eqnarray*}
	\mathcal{Y}(t)&=& g(\mathcal{X}_T)+\int^{t}_{0}f(s,\mathcal{X}_s,\mathcal{Y}_s,0)ds,
\end{eqnarray*}	
and the distribution of $Y^{\varepsilon}$ verifies a large deviation principle.	
	
The rest of the paper is organized as follows: Section 2 is dedicated to the formulation of the problem of  FBSDEs with time delayed generator, likewise preliminary results are given. Section 3 is dedicated to derive the main results.  Paper finishes with a list of references.
	
	\section{Formulation of the problem}
	For a strict positive real number $T$, let us consider $(\Omega,\mathcal{F},\P,(\mathcal{F}_t)_{0\leq t\leq T})$ a filtered probability space, where the filtration $(\mathcal{F}_t)_{0\leq t\leq T}$ is assumed to be complete, right continuous and generated by a $(W_t)_{0\leq t\leq T}$, a one-dimensional Brownian motion. In the present paper, we consider the following coupled forward-backward stochastic differential delayed equations (FBSDDEs in short): For each $T>0$ and $s\in [0,T]$.
	\begin{eqnarray}\label{Eq1}
		X^{\epsilon}(s)&=&x+\int^s_0 b(r,X_r^{\epsilon})dr+\sqrt{\epsilon}\int^s_0 \sigma(r,X_r^{\epsilon})dW(r),\nonumber\\\label{FBSDDE}\\\nonumber
		Y^{\epsilon}(s)&=& g(X_T^{\epsilon}) + \int_s^T f(r,X_r^{\epsilon }, Y_r^{\epsilon},Z_r^{\epsilon})dr - \int_s^T Z^{\epsilon }(r) d W(r)
	\end{eqnarray}
	where $X_{s}^{\epsilon}=(X^{\epsilon}(s+u))_{-T\leq u\leq 0}$, $Y_{s}^{\epsilon}=(Y^{\epsilon}(s+u))_{-T\leq u\leq 0}$ and $Z_{s}^{\epsilon}=(Z^{\epsilon}(s+u))_{-T\leq u\leq 0}$ designed respectively all the past of the processes $X^{\epsilon}, Y^{\epsilon}$ and $Z^{\epsilon}$ until $s$. Next, we extend the solution on $[-T,0]$ assuming that, $X^{\epsilon}(s)=X^{\epsilon}(0)=x, Y^{\epsilon}(s)= Y^{\epsilon}(0)$ and $Z^{\epsilon}(s)=0$, for $s\in[-T,0]$. In order to give what we mean by solution of \eqref{Eq1}, let us set the following spaces.
	\begin{description}
		\item $\bullet$ Let $L_{-T}^2 (\mathbb{R})$ denote the space of measurable functions $ z : [-T;0] \rightarrow \mathbb{R} $ satisfying
		$$   \int_{-T}^0 \mid z(t) \mid^2 dt < +\infty .
		$$
		\item $ \bullet $ Let $ L_{-T}^{\infty } (\mathbb{R} )$ denote the space of bounded, measurable functions $y: [-T,0] \rightarrow \mathbb{R} $\\
		satisfying
		$$
		\sup\limits_{-T\leq t\leq 0} \mid y(t) \mid^2 < +\infty.
		$$
		
		\item $\bullet$ Let $ L^2(\Omega,\mathcal{F}_T,\mathbb{P})$ be the space of $\mathcal{F}_T$-measurable random variables $\xi: \Omega \rightarrow \mathbb{R} $ normed by $$\|\xi\|_{L^2}^2=\E(|\xi|^2)$$.
		\item \item $\bullet$ Let $ \mathcal{S}^2(\R)$ denote the space of all predictable process $\eta$ with values in $\R$ such that $$\E\left(\sup_{0\leq s\leq T}e^{\beta s}|\eta(s)|^2\right)<+\infty.$$
		\item $\bullet$ Let $\mathcal{H}^2(\R^n)$ denote the space of all predictable process $\eta$ with values in $\R$ such that $$\E\left(\int_{0}^Te^{\beta s}|\eta(s)|^2ds\right)<+\infty.$$
	\end{description}
	\begin{defi}
		A triple of processes $(X^{\epsilon},Y^{\epsilon},Z^{\epsilon})$ is called an adapted solution of \eqref{Eq1}, if $(X,Y,Z)$ belongs to $\mathcal{H}^2(\R)\times\mathcal{S}^{2}(\R)\times \mathcal{H}^2(\R^n)$, and satisfies \eqref{Eq1} $\P$-almost surely.	
	\end{defi}
	
	Our first aims of this paper, is to derive an existence and uniqueness result for FBSDDE \eqref{FBSDDE} under the following assumptions.
	\begin{description}

		\item {(\bf A1)} $ \phi: \Omega \times [0,T]\times L_{-T}^2 (\mathbb{R}) \rightarrow \mathbb{R} $ is a product measurable and $ {\bf F} $-adapted function such that there exists a probability measure $\alpha$ defined on $([-T,0],\mathcal{B}([-T,0]))$ and a positive constant $K$  satisfying, for $\P \otimes \lambda $-a.e. $(\omega ,t) \in \Omega \times[0,T]$ and for any $ x_t, x'_t \in  L_{-T}^2 (\mathbb{R})$ 
		\begin{itemize}
			\item [(i)]
			$\displaystyle
			\mid \phi(t, x_t) - \phi(t,x'_t )\mid^2 
			\leq  K \int_{-T}^0  \mid x(t+u) - x'(t+u) \mid^2  \alpha (du)$,
			
			\item [(ii)]For $t<0,\; \; \phi(t,x_t)= 0 $,
			\item [(iii)]$\displaystyle \mathbb{E} \left[ \int_{0}^T \vert \phi(t,0) \vert^2 dt\right]  <  +\infty$, 
		\end{itemize}
		where $\phi=b,\sigma$.
		\item {(\bf A2)} $ g: \Omega \times [0,T]\times L_{-T}^2 (\mathbb{R}) \rightarrow \mathbb{R} $ is a product measurable and $ {\bf F} $-adapted function such that there exists a probability measure $\alpha$ defined on $([-T,0],\mathcal{B}([-T,0]))$ and a positive constant $K$  satisfying
		$$\displaystyle |g(x_t)- g(x'_t)|^2\leq K\int_{-T}^{0}|x(t+u)-x'(t+u)|^2\alpha(du).$$
		\item{(\bf A3)}	 $f: \Omega \times [0,T]\times L_{-T}^2 (\mathbb{R})\times  L_{-T}^{\infty} (\mathbb{R} ) \times  L_{-T}^2 (\mathbb{R}^n ) \rightarrow \mathbb{R} $  is a product measurable and $ {\bf F} $-adapted function such that there exists a probability measure $\alpha$ defined on $([-T,0],\mathcal{B}([-T,0]))$ and a positive constant $K$  satisfying for $u_t= (x_t, y_t, z_t)$ and $u'_t= (x'_t, y'_t, z'_t)$:
		\begin{itemize}
		\item [(i)]
		$\|f(u_t)- f(u'_t)|^2\leq K\int_{-T}^0|u(t+v)- u'(t+v)|^2 \alpha (dv)$,	
		\item [(ii)]For $t<0,\; \; f(t,u_t)= 0 $,
		\item [(iii)]$\displaystyle \mathbb{E} \left[ \int_{0}^T \vert f(t,\textbf{0}) \vert^2 dt\right]  <  +\infty.$ 
	\end{itemize}
		
	\end{description}
	\begin{rem}\label{R1}
		\begin{itemize}
			\item [(a)] Assumption $(\bf A1)$-$(ii)$ and $(\bf A3)$-$(ii)$  allow us to take $(X^{\epsilon}(t),Y^{\epsilon}(t),Z^{\epsilon}(t))=(X^{\epsilon}(0),Y^{\epsilon}(0),0)$ for $t<0$, as a solution of \eqref{Eq1}.
			\item [(b)] The quantity $\phi(t,0)$ in $ (\bf A1)$-$(iii) $ should be understood as a value of the generator $\phi$ at $x_t=0$.
			\item [(c)]  The quantity $f(t,\textbf{0})$ in $ (\bf A3)$-$(iii) $ should be understood as a value of the generator $f$ at $u_t=(0,0,0)$
		\end{itemize}
	\end{rem}
\begin{propo}
	Assume that (\textbf{A1} ), (\textbf{A2}) and (\textbf{A3} ) hold. For $T$ or $K$ small enough such that \[ 8K e \max (1, T) < 1,\] there exists an unique adapted solution $(X^{\epsilon}, Y^{\epsilon}, Z^{\epsilon} )$ for equation \eqref{Eq1}. 
\end{propo}
\begin{proof} 
We refer the reader to the works of Delong and Imkeller \cite{DI1} for the proof of existence solution of the backward BSDE appears in \eqref{Eq1}.
\\ Let us prove the existence result for the forward SDE one.

	To prove the existence of a solution of the equation \eqref{Eq1}, we follow the classical idea by constructing a Picard scheme and its convergence. \\
For a fix $t_0$,	let $X^{\epsilon,0}(s)=0, \forall \ s \in [0, t_0]$ and define recursively for $n \in \mathbb{N}$ and $s \in [0, t_0]$,
	\begin{equation}\label{ea}
		X^{\epsilon,n+1}(s) = x + \int_{0}^{s}b(r,X^{\epsilon,n}_{r})dr +\sqrt{\epsilon} \int_{0}^{s}\sigma(r,X^{\epsilon,n}_{r})dW(r).
	\end{equation}
	We then have:
	\begin{equation*}
		X^{\epsilon,n+1}(s)- X^{\epsilon,n}(s) =  \int_{0}^{s}\left[ b(r,X^{\epsilon,n}_{r})- b(r,X^{\epsilon,n-1}_{r})\right] dr + \sqrt{\epsilon} \int_{0}^{s}\left[ \sigma(r,X^{\epsilon,n}_{r}) - \sigma(r,X^{\epsilon,n-1}_{r})\right] dW(r) .
	\end{equation*}
	Using Young's inequality, we have 
	\begin{eqnarray}\label{ey}
	\sup_{0 \leq s \leq t_0}	\mid X^{\epsilon,n+1}(s)- X^{\epsilon,n}(s)\mid^{2} 
			 &\leq & 2\sup_{0 \leq s \leq t_0}\mid \int_{0}^{s}\left[ b(r,X^{\epsilon,n}_{r})- b(r,X^{\epsilon,n-1}_{r})\right] dr\mid^{2}\notag \\ && + 2\sup \limits_{0 \leq s \leq t_0}\mid \sqrt{\epsilon}  \int_{0}^{s}\left[ \sigma(r,X^{\epsilon,n}_{r}) - \sigma(r,X^{\epsilon,n-1}_{r})\right] dW(r)\mid^{2}.\notag \\ 
	\end{eqnarray}
	Using Holder's inequality, assumption \textbf{(A1)} and Fubini's theorem respectively we have                                                 
	\begin{eqnarray}\label{eh}
&&	\E 	\left[ \sup_{0 \leq s \leq t_0}\mid \int_{0}^{s}\left[ b(r,X^{\epsilon,n}_{r})- b(r,X^{\epsilon,n-1}_{r})\right] dr\mid^{2}\right]\notag \\ 
		 &\leq & 2T\E \left[\int_{0}^{t_0}\mid b(r,X^{\epsilon,n}_{r})- b(r,X^{\epsilon,n-1}_{r})\mid^{2}dr \right] \notag \\
		 &\leq & 2 T K \E \left[\int_{0}^{t_0} \int_{-T}^{0}|X^{\epsilon,n}(r+u)-X^{\epsilon,n-1}(r+u)|^2 \alpha(du)dr\right] \notag\\ 
		  &\leq & 2 T K \E \left[ \int_{0}^{t_0} |X^{\epsilon,n}(r)-X^{\epsilon,n-1}(r)|^2 dr\right].
	\end{eqnarray}	
	On the other hand, using Doob's inequality, assumption \textbf{(A1)} and Fubini's theorem respectively we have
	\begin{eqnarray}\label{eg}
		&&\E \left[\sup \limits_{0 \leq s \leq t_0}\mid\sqrt{\epsilon}  \int_{0}^{s}\left[ \sigma(r,X^{\epsilon,n}_{r}) - \sigma(r,X^{\epsilon,n-1}_{r})\right] dW(r)\mid^{2}\right] \notag \\
		&& \leq 4 \epsilon\E\left[\int_{0}^{t_0}\mid \sigma(r,X^{\epsilon,n}_{r}) - \sigma(r,X^{\epsilon,n-1}_{r})\mid^{2}dr\right] \notag \\ &\leq & 4 K \epsilon\E \left[ \int_{0}^{t_0} \int_{-T}^{0}|X^{\epsilon,n}(r+u)-X^{\epsilon,n-1}(r+u)|^2 \alpha(du)dr\right] \notag\\ 
		&\leq & 4  K \epsilon \E \left[ \int_{0}^{t_0} |X^{\epsilon,n}(r)-X^{\epsilon,n-1}(r)|^2 dr\right].
	\end{eqnarray}
Taking expectation in \eqref{ey} together with \eqref{eh} and \eqref{eg} we get 
	\begin{eqnarray*}
		&&\E\left[ \sup \limits_{0 \leq s \leq t_0}\mid X^{\epsilon,n+1}(s)- X^{\epsilon,n}(s)\mid^{2}\right]  \\
		&& \leq 2K(T+ 4\epsilon)\E\left[  \int_{0}^{t_0} |X^{\epsilon,n}(r)-X^{\epsilon,n-1}(r)|^2 dr \right] \\ 
		&& \leq 2Kt_0(T+ 4 \epsilon)\E\left[ \sup_{0 \leq s \leq t_0} |X^{\epsilon,n}(s)-X^{\epsilon,n-1}(s)|^2  \right]. 
	\end{eqnarray*}
	
By iteration method we obtain
	\begin{eqnarray}\label{ei}
		&&\E\left[ \sup \limits_{0 \leq s \leq t_0}\mid X^{\epsilon,n+1}(s)- X^{\epsilon,n}(s)\mid^{2}\right] \notag \\
		&& \leq \Big(2Kt_0(T+ 4\epsilon)\Big)^n\E\left[ \sup_{0 \leq s \leq t_0} |X^{\epsilon,1}(s)|^2  \right]. 
	\end{eqnarray}
Since it's not difficult to prove that $\displaystyle\E\left[ \sup_{0 \leq s \leq t_0} |X^{\epsilon,n}(s)|^2  \right]< + \infty$ for all $n$. In this fact, taking  $ t_{0} \in [0,T] $ such that $2Kt_{0}(T+4 \epsilon) < 1$,  it follows from  \eqref{ei} that $\{X^{\epsilon, n}(t), \ 0\leq t  \leq t_0\}_{n \in \N}$ is a Cauchy's  sequence on $\mathcal S^2(\R)$. 
Therefore there exists a process $(X^{\epsilon}(t))_{0\leq t  \leq t_0}$ limit of $\{X^{\epsilon, n}(t), \ 0\leq t  \leq t_0\}_{n \in \N}$ in  $\mathcal S^2(\R)$.

By passing to the limit in \eqref{ea} we obtain 
	\begin{equation*}
		X^{\epsilon}(t) = x + \int_{0}^{t}b(s,X_{s}^{\epsilon})ds + \sqrt{\epsilon}\int_{0}^{t}\sigma(s,X_{s}^{\epsilon})dW(s) \; ; \; 0 \leq t \leq t_{0}.
	\end{equation*}
If $t_0= T$ the proof of existence is complete. But if $t_0 \neq T$ we need to prove an existence result of forward SDE 
	\begin{equation}\label{eb}
	X^{\epsilon}(t) = X^{\epsilon}(t_0) + \int_{t_0}^{t}b(s,X_{s}^{\epsilon})ds + \sqrt{\epsilon}\int_{t_0}^{t}\sigma(s,X_{s}^{\epsilon})dW(s) \; ; \; t_0 \leq t \leq T.
\end{equation}
Using the same method as above, we can find a $t_1$ such that $2K(t_1- t_0)(T+ 4 \epsilon) < 1$ and there exist a process $(X^{\epsilon}(t))_{t_0\leq t  \leq t_1}$ solution of the forward SDE of \eqref{eb}. If $t_1= T$ the proof is finish. Otherwise, we repeat the above processes. Thus we get a sequence $(t_i)_{i\geq 0} $ such that $0 < t_0 < t_1 < ...< t_p\leq T$.
 One can find  $p \in \N$ (for exemple $p=[2K T(T+4 \epsilon)] +1$ )  such that $t_p= T$ and    $2K(T- t_{p-1})(T+ 4 \epsilon) < 1.$
 Finally we obtain the existence result for the forward SDE of \eqref{Eq1} on the time interval $[0, T]$.

For the uniqueness result, let consider $X^{\epsilon,1}$ and $X^{\epsilon,2}$ two solutions of the forward SDE of \eqref{Eq1}. It follows by the same computation used in the proof of existence that: for all $s\in [0, T]$ 
\begin{eqnarray*}
	\E\left[ \mid X^{\epsilon,1}(s)- X^{\epsilon,2}(s)\mid^{2}\right]   \leq 2KT(T+ 4 \epsilon)\E\left[ \int_{0 }^T |X^{\epsilon,1}(r)-X^{\epsilon,2}(r)|^2 dr \right]. 
\end{eqnarray*}
We end by using Gronwall's Lemma and obtain $X^{\epsilon,1}= X^{\epsilon,2}$ a.s.
	
\end{proof}
Let us end this section deriving the following  needed result.
\begin{lem}\label{l0}
	Assume that \textbf{(A1)} and \textbf{(A2)} hold. Then we have 
\begin{eqnarray}
&& 	\E \left[ \sup_{ 0\leq s \leq T} |X^{\epsilon}(s) |^2 \right] \leq C_2 (1+ |x|^2), \notag \\ 
&& \E \left[ \sup_{ 0\leq s \leq T} |Y^{\epsilon}(s) |^2  \right]\leq C_2 (1+ |x|^2),   \notag \\
&& \E \left[ \int_{ 0}^T |Z^{\epsilon}(s) |^2ds\right] \leq C_2 (1+ |x|^2),   
\end{eqnarray} where $C_2$ is a positive constant indepnedent of $\epsilon$ and $t$	
\end{lem}
\begin{lem}\label{l1}
	Assume that \textbf{(A1)} and \textbf{(A2)} hold. Pick $0< \epsilon_2 < \epsilon_1 < 1$. Then we have 
	\begin{eqnarray}
	& (i)\ \ 	\E \left(  \displaystyle{\sup_{0 \leq s \leq T}} | X^{\epsilon_1} (s) - X^{\epsilon_2} (s)|^2\right) \leq  C (\sqrt{\epsilon_1} - \sqrt{\epsilon_2})^2\notag\\ \notag\\ & (ii)\ \  \E \left(  \displaystyle{\sup_{0 \leq s \leq T}} | Y^{\epsilon_1} (s) - Y^{\epsilon_2} (s)|^2\right) \leq  C (\sqrt{\epsilon_1} - \sqrt{\epsilon_2})^2 \\ \notag \\  & (iii)\ \  \E \left( \displaystyle \int_0^T| Z^{\epsilon_1} (s) - Z^{\epsilon_2} (s)|^2 ds\right) \leq  C (\sqrt{\epsilon_1} - \sqrt{\epsilon_2})^2\nonumber  
	\end{eqnarray}
where $C$ is a constant independent of $\epsilon$ and $T$.
\end{lem}
\begin{proof} It's not difficult to prove that  
	 $(ii)$ and $(iii)$ follow from $(i)$. Indeed, one can show that there exists a constant $C> 0$ such that 
	 \begin{eqnarray*}
	 	  &&\E \left( \sup_{0\leq s \leq T} e^{\beta s}|Y^{\epsilon_1}(s) - Y^{\epsilon_2}(s)|^2+ \int_0^Te^{\beta s} |Z^{\epsilon_1}(s) - Z^{\epsilon_2}(s)|^2 ds\right) \\ &\leq &  C \E \Big( \sup_{0 \leq s \leq T}e^{\beta s}|X^{\epsilon_1}(s)- X^{\epsilon_2}(s)|^2\Big).\notag
	 \end{eqnarray*}
	Therefore  we need to prove  only $(i)$.  For this let us set $\Delta X(s)=  X^{\epsilon_1} (s) - X^{\epsilon_2} (s)$ and denote $C$ a positive  constant independent of $T$ and $\epsilon$ which make vary line to line.   \\ Applying It\^o formula  to $|\Delta X(s)|^2$ together with assumption \textbf{(A1)},  Fubini's theorem and Gronwall's lemma we get
	 \begin{eqnarray} \label{eq3}
	 \E \left( |\Delta X(s)|^2 \right)	 & \leq & 2 \E \left[\int_0^s |\Delta X(r)|^2 dr \right] + K(2+ \epsilon_1) \E \left[\int_{-T}^0 | \Delta X(r+ u)|^2 dr \alpha (du) \right] \notag \\ && + 2K (\sqrt{\epsilon_1} - \sqrt{\epsilon_2})^2 \E \left[ \int_{-T}^0 \int_0^s | X(r+u)|^2 dr \ \alpha (du)  \right] \notag \\ && + 2(\sqrt{\epsilon_1} - \sqrt{\epsilon_2})^2 \E \left[ \int_0^s | \sigma (r, 0)|^2 dr\right]  \notag\\ & \leq & \Big(2+ K(2+ \epsilon_1)\Big)\E \left[ \int_0^s |\Delta X(r)|^2 dr  \right] + C (\sqrt{\epsilon_1}- \sqrt{\epsilon}_2)^2 \notag \\ & \leq & C (\sqrt{\epsilon_1}- \sqrt{\epsilon}_2)^2 
	 \end{eqnarray}
 Using again It\^o formula to $|\Delta X(s)|^2$ and taking the supremum, it follows from BDG inequality and \eqref{eq3} that 
 \begin{eqnarray*}
 	 \E \Big(\sup_{0 \leq s \leq T} |\Delta X(s)|^2 \Big)	\leq C (\sqrt{\epsilon_1}- \sqrt{\epsilon}_2)^2 .  
 \end{eqnarray*}   
\end{proof}
Now consider the following deterministic equations
\begin{eqnarray}\label{eq2}
		\mathcal X(s)&=& x + \int_0^s b(r,\mathcal X_r) dr \notag \\
	\mathcal Y(s)&=& g(\mathcal X_T) + \int_s^T f(r,\mathcal X_r,\mathcal Y_r, 0) dr.
\end{eqnarray} 
\begin{rem}
	If $\epsilon=0$ then \eqref{Eq1} becomes \eqref{eq2}. 
\end{rem} 
In view of assumptions  \textbf{(A1)}-\textbf{(A3)} and Lemma \eqref{l1} we have:
\begin{propo}
Assume that \textbf{(A1)} and \textbf{(A2)} hold, then there exists a unique solution $(	\mathcal X, 	\mathcal Y )$ for equation  \eqref{eq2}. 
\end{propo}
	\section{Main results}
 
	\subsection{Convergence of distributions}
	\begin{theo} Under the assumptions \textbf{(A1)}-\textbf{(A3)}, we have the following results let $(X^{\epsilon}, Y^{\epsilon}, Z^{\epsilon})$, $(\mathcal{X}, \mathcal{Y})$ be respectively the unique solution of \eqref{Eq1} and \eqref{eq2} :
		\begin{itemize}
			\item [(i)] For all $\delta > 0$
			 \begin{equation}\label{eq4}
				\lim_{\epsilon \rightarrow 0} \P \left( \sup_{0 \leq s\leq T}|X^{\epsilon}(s)- \mathcal X(s)| > \delta  \right)= 0.	
			\end{equation}
			\item [(ii)] Let $Q^{\epsilon}$ be the probability measure on $ D(\R^n)$ defined by $Q^{\epsilon}(A)= \P\left(\big(Y^{\epsilon}(.)\big)^{-1}(A) \right)$. 	Then 	there exists a subsequence $Q^{\epsilon_n}$ of $Q^{\epsilon}$ and a probability law $Q$ on $D(\R^n)$ such that $Q^{\epsilon_n}$ converges weakly in the Meyer-Zheng topology to $Q$ as $n \rightarrow + \infty$.
		\end{itemize}
	\end{theo}
\begin{proof}
	\begin{eqnarray*}
		| X^{\epsilon}(s)-  \mathcal{X}(s) |  &\leq & \left|  \int_0^s [b(r, X_r^{\epsilon}) - b(r, \mathcal X_r)]dr\right| + \sqrt{\epsilon}\left| \int_0^s\sigma(r, X_r^{\epsilon})dW(r)\right|\\
		\sup_{0 \leq s \leq T} 	\left| X^{\epsilon}(s)-  \mathcal{X}(s) \right| &\leq & 	\sup_{0 \leq s \leq T} \left|  \int_0^s [b(r, X_r^{\epsilon}) - b(r, \mathcal X_r)]dr\right| + \sqrt{\epsilon} \sup_{0 \leq s \leq T}  \left| \int_0^s\sigma(r, X_r^{\epsilon})dW(r)\right|
	\end{eqnarray*}
\begin{eqnarray}\label{eq5}
	 \P \left(\sup_{0 \leq s \leq T}  | X^{\epsilon}(s)-  \mathcal{X}(s) |> \delta \right) & \leq & \P \left(\Big|\int_0^T [b(r, X_r^{\epsilon}) - b(r, \mathcal{X}_r)] dr   \Big|> \dfrac{\delta}{2} \right) \notag\\  & +& \P \left(\sqrt{\epsilon}\sup_{0 \leq s \leq T}\Big|\int_0^T \sigma(r, X_r^{\epsilon}) dw(r)   \Big|> \dfrac{\delta}{2} \right)
\end{eqnarray}
\begin{eqnarray}\label{eq6}
	\P \left(\Big|\int_0^T [b(r, X_r^{\epsilon}) - b(r, \mathcal{X}_r)] dr   \Big|> \dfrac{\delta}{2} \right) & \leq &\dfrac{4}{\delta^2} \E \left[ \Big| \int_0^T b(r, X_r^{\epsilon})- b(r, \mathcal X_r) dr \Big|^2\right] \notag \\ & \leq & \dfrac{4K}{\delta^2} \E \Big[  \int_0^T|X^{\epsilon}(r)- \mathcal{X}(r)|^2 dr \Big]\notag \\ & \leq & C\times  \E \Big[  \sup_{0 \leq s \leq T  }| X^{\epsilon}(r)-\mathcal{X}(r)|^2\Big]\notag \\ &\leq & C \epsilon. 
\end{eqnarray}
\begin{eqnarray}\label{eq7}
	  \P \left(\sqrt{\epsilon}\sup_{0 \leq s \leq T}\Big|\int_0^T \sigma(r, X_r^{\epsilon}) dw(r)   \Big|> \dfrac{\delta}{2} \right) & \leq &  \dfrac{4 \epsilon}{\delta^2} \E\Big[\sup_{ 0\leq s \leq T} \Big| \int_0^T \sigma (r, X_r^{\epsilon})dW(r)\Big|^2\Big]\notag \\  & \leq &   \dfrac{4 \epsilon}{\delta^2} \E \left[ \int_0^T | \sigma (r, X_r^{\epsilon})|^2 dr\right]\notag \\ & \leq &  \dfrac{4 \epsilon}{\delta^2}\times K \E \left[ \int_0^T |  X^{\epsilon}(r)|^2 dr\right]+ \dfrac{4 \epsilon}{\delta^2}\times  \E\left[ \int_0^T |\sigma (r, 0)|^2 dr\right]\notag \\ &\leq & C\epsilon.
\end{eqnarray}
Combining \eqref{eq6} and \eqref{eq7} with  \eqref{eq5}  we have 
\begin{eqnarray*}
	\P \left(\sup_{0 \leq s \leq T}  | X^{\epsilon}(s)-  \mathcal{X}(s) |> \delta \right) & \leq & C \epsilon.
\end{eqnarray*} which end the proof of $(i)$.

For a given subdivision $\pi$:$0=t_{0}<t_{1}<....t_{n}=T$, we get
\begin{eqnarray*}
	&&V_{T}^{\pi}(Y^{\epsilon})\\
	&& = \E\left( \mid g(X_{T}^{\epsilon})\mid\right) + \sum_{k=0}^{n-1}\E\left[\mid \E\left[Y^{\epsilon}(t_{k+1})-Y^{\epsilon}(t_{k}) \right]|\mathcal{F}_{t_{k}} \mid \right]\\  
	&& = \E\left( \mid Y^{\epsilon(t_{n})}\mid\right) + \sum_{k=0}^{n-1}\E\left[\mid \E\left[\int_{t_{k}}^{t_{k+1}}f(r,U_{r}^{\epsilon})dr\right]|\mathcal{F}_{t_{k}} \mid \right]\\
	&& \leq \E\left( \mid g(X_{T}^{\epsilon})\mid\right) + \sum_{k=0}^{n-1}\E\left[\int_{t_{k}}^{t_{k+1}}\mid f(r,U_{r}^{\epsilon})\mid dr \right]\\
	&& \leq \E\left( \mid g(X_{T}^{\epsilon})-g(0)\mid + \mid g(0)\mid\right) + \sum_{k=0}^{n-1}\E\left[\int_{t_{k}}^{t_{k+1}}\mid f(r,U_{r}^{\epsilon})-f(r,0)\mid dr \right]  + \E\left[\int_{t_{k}}^{t_{k+1}}\mid f(r,0)\mid dr \right] 
\end{eqnarray*}
Using assumptions \textbf{(A1)}, we then have
\begin{eqnarray*}
	&&V_{T}^{\pi}(Y^{\epsilon})\\
	&& \leq \sqrt{K}\E\left[ \int_{-T}^{0}\mid X^{\epsilon}(T+u)\mid \alpha(du)\right]  + |g(0)| \\
	&& +  \sum_{k=0}^{n-1}\left(\sqrt{K}\E\left[\int_{t_{k}}^{t_{k+1}}\int_{-T}^{0}\mid U^{\epsilon}(r+u) \mid\alpha(du) dr \right] + \E\left[\int_{t_{k}}^{t_{k+1}}\mid f(r,0)\mid dr \right] \right)\\
	&& \leq \sqrt{K}\E\left[ \int_{-T}^{0}\mid X^{\epsilon}(T+u)\mid \alpha(du)\right]  + |g(0)| + \E\left[ \int_{0}^{T}\mid f(r,0)\mid dr  \right]\\                                                            && + \sqrt{K}\E\left[\int_{0}^{T}\int_{-T}^{0}\left(\mid X^{\epsilon}(r+u)| + |Y^{\epsilon}(r+u)|+ |Z^{\epsilon}(r+u) |\right)\alpha(du) dr\right] 
\end{eqnarray*}
Applying Fubini's theorem, changeing the variable, taking the supremum and Jensen's inequality, we have
\begin{eqnarray*}
	V_{T}^{\pi}(Y^{\epsilon})
	& \leq &\sqrt{K}\E\left[ \sup \limits_{0 \leq t \leq T}\mid X^{\epsilon}(t)\mid \right]  + |g(0)|+ \E \Big[  \int_{0}^{T}\mid f(r,0)\mid dr \Big]\\                                            
	    && + \sqrt{K}\E\left( T\sup_{0 \leq t \leq T} \mid X^{\epsilon}(t) \mid + T\sup_{0 \leq t \leq T} \mid Y^{\epsilon}(t)|+ \int_{0}^{T}\mid Z^{\epsilon}(r)\mid dr  \right)\\ && \leq  C \max(1, T)  \left[1+ \E\left( \sup_{0 \leq t \leq T} \mid X^{\epsilon}(t) \mid^2 + \sup_{0 \leq t \leq T} \mid Y^{\epsilon}(t)|^2+ \int_{0}^{T}\mid Z^{\epsilon}(r)\mid^2 dr  \right)\right].
\end{eqnarray*}
Finally, it follows from Lemma \eqref{l0} that $	V_{T}(Y^{\epsilon})=\sup_{\pi}	V_{T}^{\pi}(Y^{\epsilon}) < + \infty.$
 
Now since $D(\R^{n})$ is a separable metric space, there exists a compact metric space $K$ such that $D(\R^{n})$is a subset of $K$. Note that $D(\R^{n})$ is a Lusin space: for every embeddingin a compact metric space $K$, $D(\R^{n})$ is a Borel set in $K$:$D(\R^{n}) \in \mathcal{B}(K)$. On the compact metric space $K$ we define  
\begin{equation*}
	\tilde{Q}^{\epsilon}(A) = Q^{\epsilon}(A\cap D(\R^{n})),\; A \in \mathcal{B}(K).
\end{equation*}

Clearly, $A \cap D(\R^{n})$ belongs to $\mathcal{B}(D(\R^{n})) \mathcal{D}(\R^{n})$, the last equality being true in view of Lemma 3.2. The set of probability measures on the compact metric space $K$ is compact for the weak convergence
. Hence, we can choose a subsquence also denoted $(\epsilon_{n})_{n \geq 1}$, and a probability measure $\tilde{Q}$ on $K$ such that  
\begin{equation*}
	\tilde{Q}^{\epsilon_{n}} \rightarrow^{w} \tilde{Q} \;on \; K.
\end{equation*}
We now show that
\begin{equation*}
	\tilde{Q}(D(\R^{n})) = 1.
\end{equation*}
We notice that $\tilde{Q}^{\epsilon}$ is the distribution of $Y^{\epsilon}$ considered as a random variable with values in $(K,\mathcal{B}(K))$. Furthermore, by proposition 3.1, we know that, possibly along a subsequence, $\tilde{Q}^{\epsilon_{n}}$ converges weakly to a probability law $Q^{*} \in \mathcal{M}(D(\R^{n}))$. The uniqueness of the weak limit implies that  
\begin{eqnarray*}
	Q^{*}(A)= \tilde{Q}(A)&,& \forall \; A \; \in \; \mathcal{B}(D(\R^{n})),
\end{eqnarray*}
In particular,
\begin{eqnarray*}
	&& 1 = Q^{*}(D(\R^{n})) = \tilde{Q}(D(\R^{n})).
\end{eqnarray*}
The proof is complete.

\end{proof}
	\subsection{Convergence almost surely }
	\begin{theo}
		 Under the assumptions \textbf{(A1)}-\textbf{(A3)},  $(X^{\epsilon}(s),Y^{\epsilon}(s),Z^{\epsilon}(s))_{0\leq s \leq T}$ solution of \eqref{Eq1} converge in $ \Big( \mathcal{S}^2(\R)\times  \mathcal{S}^2(\R) \times  \mathcal{H}^2(\R^n)\Big)$	  to $(\mathcal X(s), \mathcal Y(s), 0)_{0\leq s \leq T}$ solution of \eqref{eq2}, when $\epsilon \rightarrow 0$.
	\end{theo}
\begin{proof}According  Lemma \ref{l1}
$\Big(X^{\epsilon}(s),Y^{\epsilon}(s),Z^{\epsilon}(s)\Big)_{0\leq s \leq T}$ form a Cauchy sequence  and therefore converges in  $ \Big( \mathcal{S}^2(\R)\times  \mathcal{S}^2(\R) \times  \mathcal{H}^2(\R^n)\Big)$.	

  Let $\Big(\mathcal X(s), \mathcal Y(s),  \mathcal Z(s)\Big)_{0\leq s \leq T} $ be its limit. It remain to prove that   $\Big(\mathcal X(s), \mathcal Y(s),  \mathcal Z(s)\Big)_{0\leq s \leq T}$ is a unique solution of \eqref{eq2} such that $ \mathcal Z= 0$.  

Since $\mathcal X$ is a limit of $X^{\epsilon}$ in $\mathcal{S}^2(\R)$, there exists a subsequence $X^{\epsilon_n}$ converging almost surely to $\mathcal X$. Moreover as $b$ is continuous, we obtain by passing in the limit of the forward component of \eqref{Eq1} that $\mathcal X$ satisfies this equation.  
	  \begin{eqnarray*}
	  	\mathcal X(s)= x + \int_0^s b(r, \mathcal X_r) dr.
	  \end{eqnarray*} 
  On other hand, recalling that $ (\mathcal Y, \mathcal Z)$ is a limit of $(Y^{\epsilon},  Z^{\epsilon})$ in $\Big(\mathcal{S}^2(\R) \times  \mathcal{H}^2(\R^n)\Big)$,  there also exists a subsequence $ (Y^{\epsilon_n},  Z^{\epsilon_n})$ converging almost surely to $ (\mathcal Y, \mathcal Z)$ and $$\int_s^T Z^{\epsilon_n}(r) dW(r) \rightarrow 0,  \ \ \P-a.s.$$
  Therefore $\mathcal Z= 0$ and   
   passing to the limit in the backward component of \eqref{Eq1}, it follows  from the continuity of $g$ and $f$  
    that $\mathcal Y$ satisfies equation 
  \begin{eqnarray*}
  	\mathcal Y(s)= g(\mathcal X_T) + \int_s^T f(r, \mathcal X_r, \mathcal Y_r, 0 ) dr.
  \end{eqnarray*} 

\end{proof}
	\subsection{Large deviation principle}
	Let $ \big(X^{\epsilon}, Y^{\epsilon}, Z^{\epsilon}\big)_{ \epsilon \in (0,1] }  $ be the solution of the following equation:
	\begin{eqnarray}
		&& 	X^{\epsilon}(s)=x+\int^s_0 b(r,X_r^{\epsilon})dr+\sqrt{\epsilon}\int^s_0 \sigma(r,X_r^{\epsilon})dW(r)\label{sde}\\
		 &&			Y^{\epsilon}(s)=g(X_T^{\epsilon})+\int^T_s f(r,X_r^{\epsilon},Y_r^{\epsilon},Z^{\epsilon}_r )dr- \int_s^T Z^{\epsilon}(r)dW(r)\label{bsde}.
	\end{eqnarray}
	In this subsection, we study the Freidlin-Wentzell's large deviation principle for the laws of the family of processes $\big(X^{\epsilon}, Y^{\epsilon}\big)_{ \epsilon \in (0,1] }$  
		as $\epsilon \rightarrow 0$.

	In this fact, let us recall the following definitions.
\begin{defi}
	If $E$ is a complete separable metric space, then a function $I$ defined on $E$ is called a rate function if it has the following properties:
		\begin{itemize}
		\item[(i)]\ \ $I: E \rightarrow [0, + \infty], \ \ I $ is lower semicontinuous 
		\item[(ii)] If  $\; 0 \leq a \leq +\infty $ then $ C_y(a)= \{y \in E, \ \ I(y) \leq a\}$ is compact.
	\end{itemize}
\end{defi}
\begin{defi}
	 If $E$ is a complete separable metric space, $\mathcal{B}$ is the Borel $\sigma-$field on $E$, $\{\mu_{\epsilon}: \ \ \epsilon >0\},$ is a family of probability measures on $(E,\mathcal{B})$, and $I$ is a function defined on $E$ and satisfying  $(i)$-$(ii)$, then we say that $  (\mu_{\epsilon})_{\epsilon >0}$ satisfies a large deviation principle with rate $I$ if:
	 \begin{itemize}
	 \item[(a)] For every open subset $ A$ of $ E, \; \displaystyle \lim_{\epsilon \rightarrow 0}\inf\epsilon \log \mu_{\epsilon}  (A)\geq - \inf_{g \in A} I(g)$
	 \item[(b)] For every closed subset $ A$ of $ E, \; \displaystyle \lim_{\epsilon \rightarrow 0}\sup\epsilon \log \mu_{\epsilon}  (A)\leq - \inf_{g \in A} I(g)$.
	 \end{itemize}
 Here the infimum over the empty set is defined to be $+ \infty.$
\end{defi}
Firstly, let us recall the Large Deviation Principle due to Manga and Aman in \cite{Manga-AMAN}, for  forward SDE \eqref{sde}.

\begin{propo} Assume \textbf{(A1)} hold. Then
the family of processes $(X^{\epsilon})_{ 0 < \epsilon  \leq 1	} $   solution of \eqref{sde} satisfies, as $\epsilon$ tends to $0$,  the Laplace
principle (equivalently, the Large Deviation Principle) in $\mathcal{C}_x([-T, T], \R)$ with a rate function $I_1$ defined by 
\begin{eqnarray}
I_1(\phi)=\left\{ \begin{array}{ll}
	\inf\{ \frac{1}{2}\int_0^T\| \psi(t) \|^2 dt: \phi'(s)= b(s, \phi (s))+ \sigma (s, \phi(s)) \psi (s) \},  & \psi \in L^2([0, T], \R) \\\\ 
 	0 & \ \ \text{otherwise},
\end{array}	\right.   
\end{eqnarray}
where $\mathcal{C}_x([-T, T], \R)$ design the space of continuous functions $\phi$ from $[-T, T]$ to $\R$ such that $\phi(t)= x$ for all $t\in [-T, 0]$. 
\end{propo}

To establish the large deviation property for the family of the processes $(Y^{\epsilon})_{ 0 < \epsilon  \leq 1	} $ we need this importent result, due to Dembo and Zeitouni in \cite{DZ},  which allow to transfer LPD from some topology space to an other one.
\begin{lem}\label{dem}
Let $\{\mu_{\epsilon}\}$ be a family of probability measures that satisfies Large Deviation Principle with a good rate function $I$ on a Hausdorff topological space $\mathcal S$, and for $ 0< \epsilon  \leq 1$, let $F_{\epsilon}: \mathcal S\rightarrow \mathcal Q$  be a continuous function, with $(\mathcal Q, d)$ a metric space. Assume that there exists a measurable map $F: \mathcal S \rightarrow \mathcal Q$ such that for every $\alpha < + \infty,$ \[ \limsup_{\epsilon \rightarrow 0} \sup_{\{x: I(x) \leq \alpha\}} d(F_{\epsilon}(x), F(x))=0.\] Then the family of the probability measures $\{\mu_{\epsilon} \circ F_{\epsilon}^{-1}\}$ satisfy the LDP in $\mathcal Q$ with the good rate function 
\[I'(y)= \inf \{ I(x): y= F(x)\}.\] 	
\end{lem}
\begin{theo}\label{pdg}
Assume \textbf{(A1)}-\textbf{(A3)} hold. Then, 
the family $(Y^{\epsilon})_{ 0 < \epsilon  \leq 1	} $ satisfies, as $\epsilon$ tends to $0$, a large deviation principle with a rate function 
\begin{eqnarray*}
	 I_2 (\phi)= \inf\{ I_1(\varphi): 
	 F(\varphi)(s)= \phi(s)= Y^{\varphi}(s) \},
\end{eqnarray*} 
where 	$F: \mathcal{C}([-T, T], \R) \rightarrow \mathcal{C}([-T, T], \R)$ defined by $F(\psi)(s)= Y^{\psi}(s)$ with $Y^{\psi} $ satisfies
\[ Y^{\psi}(s)= g(\psi_T) + \int_s^T f(r, \psi_r, Y^{\psi}_r, 0)dr. \]
\end{theo}
\begin{proof}
Let consider $F_{\epsilon}: \mathcal{C}([-T, T], \R) \rightarrow \mathcal{C}([-T, T], \R)$ defined by: for all $ \psi= X^{\epsilon}(\omega),$ $  F_{\epsilon}(\psi)= Y^{\epsilon}(\omega)$ where $ Y^{\epsilon}$ satisfies \eqref{bsde} . According to   Lemma \ref{dem}, it suffices to prove that   $F_{\epsilon}$ is continuous and converge uniformly to $F$.  
The continuity follows in view of this inequality
\begin{eqnarray*}
	 \sup_{s \in [0, T]} | F_{\epsilon}(\varphi)(s) - F_{\epsilon} (\varphi')(s) | = \sup_{s \in [0, T]} | Y^{\epsilon}(\omega_1)(s) - Y^{\epsilon} (\omega_2)(s) |\leq \sup_{s \in [0, T]} |\varphi (s) - \varphi'(s)|.
\end{eqnarray*}
Let now show the uniform convergence of the mapping $F^{\epsilon}$. Consider  $\mathcal{K}$ a compact set of $\mathcal C ([-T, T], \R)$. It follows from defintion of $F^{\epsilon}$ and $F$, and Lemma \ref{l1} that:
\begin{eqnarray*}
	\sup_{\psi \in \mathcal K}| F_{\epsilon}( \psi)- F(\psi)|^2 &= &\sup_{\psi \in \mathcal K} \sup_{s \in [0, T]}| F_{\epsilon}( \psi)(s)- F(\psi)(s)|^2 \\ & \leq &\sup_{\psi \in \mathcal K} \sup_{s \in [0, T]}| Y^{\epsilon}( \omega_)(s)- Y(\omega)(s)|^2\\ &\leq & C \epsilon . 
\end{eqnarray*}
Letting $\epsilon$ tends to $0$ we get the uniform convergence.
\end{proof}
\begin{rem}
	This works is devoted 
\end{rem}

	\bigskip
	
	{\bf Acknowledgment:}
	This work was carried out in two stages. First, it began during the stay of the first author (Clément Manga) at "UFR Mathématique et Informatique" of Félix H. Boigny University, Cocody (Côte d'Ivoire). Then it ended during the stay of the second author (Auguste Aman) at the Mathematics Laboratory of Assane Seck University, Ziguinchor (Senegal). These two authors would like to thank the various administrations.

\end{document}